\documentclass[12pt]{article}
\usepackage{amssymb,amsthm,amsmath,latexsym}
\newtheorem{thm}{Theorem}
\newtheorem{prop}[thm]{Proposition}
\newtheorem{lem}[thm]{Lemma}

\theoremstyle{remark}
\newtheorem{rem}[thm]{Remark}

\newcommand{\ZZ}{\mathbb{Z}}
\newcommand{\RR}{\mathbb{R}}

\newcommand{\allone}{\mathbf{1}}
\DeclareMathOperator{\Aut}{Aut}

\begin{document}
\title{Extremal Unimodular Lattices in Dimension $36$}


\author{
Masaaki Harada\thanks{
Research Center for Pure and Applied Mathematics,
Graduate School of Information Sciences,
Tohoku University, Sendai 980--8579, Japan.
email: mharada@m.tohoku.ac.jp.
This work was carried out
at Yamagata University.}
}

\maketitle

\noindent
{\bf Dedicated to Professor Vladimir D. Tonchev on His 60th Birthday}

\begin{abstract}
In this paper, new extremal odd unimodular lattices in dimension $36$
are constructed.
Some new odd unimodular lattices in dimension $36$ with 
long shadows are also constructed.
\end{abstract}

\section{Introduction}
A (Euclidean) lattice $L \subset \RR^n$
in dimension $n$
is {\em unimodular} if
$L = L^{*}$, where
the dual lattice $L^{*}$ of $L$ is defined as
$\{ x \in {\RR}^n \mid (x,y) \in \ZZ \text{ for all }
y \in L\}$ under the standard inner product $(x,y)$.
A unimodular lattice is called {\em even}
if the norm $(x,x)$ of every vector $x$ is even.
A unimodular lattice, which is not even, is called
{\em odd}.
An even unimodular lattice in dimension $n$ exists if and
only if $n \equiv 0 \pmod 8$, while
an odd  unimodular lattice exists for every dimension.
Two lattices $L$ and $L'$ are {\em isomorphic}, denoted $L \cong L'$,
if there exists an orthogonal matrix $A$ with
$L' = L \cdot A$, where $ L \cdot A=\{xA \mid x \in L\}$.
The automorphism group $\Aut(L)$ of $L$ is the group of all
orthogonal matrices $A$ with $L = L \cdot A$.

Rains and Sloane~\cite{RS-bound} showed that
the minimum norm $\min(L)$ of a unimodular
lattice $L$ in dimension $n$
is bounded by
$\min(L) \le 2 \lfloor n/24 \rfloor+2$
unless $n=23$ when $\min(L) \le 3$.
We say that a unimodular lattice meeting the upper
bound is {\em extremal}.

The smallest dimension for which there is an odd unimodular
lattice with minimum norm (at least) $4$ is 
$32$ (see~\cite{lattice-database}).
There are exactly five 
odd unimodular lattices in dimension $32$ having minimum norm $4$,
up to isomorphism~\cite{CS98}.
For dimensions $33,34$ and $35$, the minimum norm of
an odd unimodular lattice is at most $3$
 (see~\cite{lattice-database}).
The next dimension for which there is an odd unimodular
lattice with minimum norm (at least) $4$ is $36$.
Four extremal odd unimodular lattices in dimension $36$
are known, namely, 
Sp4(4)D8.4 in~\cite{lattice-database},
$G_{36}$ in~\cite[Table~2]{G04},
$N_{36}$ in~\cite[Section~3]{H11} and
$A_4(C_{36})$ in~\cite[Section~3]{H12}.
Recently, one more lattice has been found, namely,
$A_6(C_{36,6}(D_{18}))$ in~\cite[Table~II]{Hodd}.
This situation motivates us to improve the number of known
non-isomorphic extremal odd unimodular lattices in dimension $36$.
The main aim of this paper is to prove the following:

\begin{prop}\label{main}
There are at least $26$ non-isomorphic extremal 
odd unimodular lattices in dimension $36$.
\end{prop}

The above proposition is established by constructing new
extremal odd unimodular lattices in dimension $36$ 
from self-dual $\ZZ_k$-codes, where
$\ZZ_{k}$ is the ring of integers modulo $k$, by using two approaches.
One approach is to consider self-dual $\ZZ_4$-codes.
Let $B$ be a binary doubly even code of length $36$
satisfying the following conditions:
\begin{align}
\label{eq:C1}
&\text{the minimum weight of $B$ is at least $16$}, \\
\label{eq:C2}
&\text{the minimum weight of its dual code $B^\perp$ is at least $4$.}
\end{align}
Then a self-dual $\ZZ_4$-code with residue code $B$
gives an extremal odd unimodular lattice in dimension $36$
by Construction A.
We show that a binary doubly even 
$[36,7]$ code satisfying the conditions (\ref{eq:C1}) and (\ref{eq:C2})
has weight enumerator
$1+ 63 y^{16}+ 63 y^{20}+ y^{36}$ (Lemma~\ref{lem:WE}).
It was shown in~\cite{PST} that there are four codes
having the weight enumerator, up to equivalence.
We construct ten new extremal odd unimodular lattices in dimension $36$
from self-dual $\ZZ_4$-codes whose residue codes are 
doubly even $[36,7]$ codes satisfying the conditions 
(\ref{eq:C1}) and (\ref{eq:C2})  (Lemma~\ref{lem:N1}).
New odd unimodular lattices in dimension $36$ with minimum norm $3$
having shadows of minimum norm $5$ are constructed from some
of the new lattices   (Proposition~\ref{prop:longS}).
These are often called unimodular lattices with long 
shadows (see~\cite{NV03}).
The other approach is to consider
self-dual $\ZZ_k$-codes $(k=5,6,7,9,19)$, which have
generator matrices of a special form given in (\ref{eq:GM}).
Eleven more new extremal odd unimodular lattices in dimension $36$
are constructed by Construction A (Lemma~\ref{lem:N2}).
Finally, we give a certain short observation on ternary self-dual codes
related to extremal odd unimodular lattices in dimension $36$.

All computer calculations in this paper were
done by {\sc Magma}~\cite{Magma}.

\section{Preliminaries} \label{sec:def}

\subsection{Unimodular lattices}
Let $L$ be an odd unimodular lattice and let $L_0$ denote the
even sublattice, that is, the
sublattice of vectors of even norms.
Then $L_0$ is a sublattice of $L$ of index $2$~\cite{CS98}.
The {\em shadow} $S(L)$ of $L$ is defined to be $L_0^* \setminus L$.
There are cosets $L_1,L_2,L_3$ of $L_0$ such that
$L_0^* = L_0 \cup L_1 \cup L_2 \cup L_3$, where
$L = L_0  \cup L_2$ and $S = L_1 \cup L_3$.
Shadows for odd unimodular lattices appeared in~\cite{CS98}
and also in~\cite[p.~440]{SPLAG},
in order to
provide
restrictions on the theta series of odd unimodular lattices.
Two lattices $L$ and $L'$ are {\em neighbors} if
both lattices contain a sublattice of index $2$
in common.
If $L$ is an odd unimodular lattice in dimension divisible by
$4$, then there are two unimodular lattices
containing $L_0$,
which are rather than $L$,
namely, $L_0 \cup L_1$ and $L_0 \cup L_3$.
Throughout this paper,
we denote the two unimodular neighbors by
\begin{equation}\label{eq:N}
Ne_1(L)=L_0 \cup L_1 \text{ and } Ne_2(L)=L_0 \cup L_3.
\end{equation}


The theta series $\theta_{L}(q)$ of $L$ is the formal power
series
$
\theta_{L}(q) = \sum_{x \in L} q^{(x,x)}.
$
The kissing number of $L$ is the second nonzero coefficient of the
theta series of $L$, that is, the number of vectors of minimum norm
in $L$.
Conway and Sloane~\cite{CS98} gave some characterization of
theta series of odd unimodular lattices and their shadows.
Using~\cite[(2), (3)]{CS98}, 
it is easy to determine
the possible theta series $\theta_{L_{36}}(q)$ and
$\theta_{S(L_{36})}(q)$ of 
an extremal odd unimodular lattice $L_{36}$
in dimension $36$ and its shadow $S(L_{36})$:
\begin{align}
\label{eq:T1}
\theta_{L_{36}}(q) =&
1 
+ (42840 + 4096 \alpha)q^4 
+(1916928 - 98304 \alpha)q^5
+ \cdots,
\\
\label{eq:T2}
\theta_{S(L_{36})}(q) =&
\alpha q  + (960  - 60 \alpha) q^3 
+ (3799296 + 1734 \alpha)q^5
+ \cdots,
\end{align}
respectively,
where $\alpha$ is a nonnegative integer.
It follows from the coefficients of $q$ and $q^3$ in 
$\theta_{S(L_{36})}(q)$ that $0 \le \alpha \le 16$.

\subsection{Self-dual $\ZZ_k$-codes and Construction A}

Let $\ZZ_{k}$ be the ring of integers modulo $k$, where $k$ 
is a positive integer greater than $1$.
A {\em $\ZZ_{k}$-code} $C$ of length $n$ 
is a $\ZZ_{k}$-submodule of $\ZZ_{k}^n$.
Two $\ZZ_k$-codes are {\em equivalent} if one can be obtained from the
other by permuting the coordinates and (if necessary) changing
the signs of certain coordinates.
A code $C$ is {\em self-dual} if $C=C^\perp$, where
the dual code $C^\perp$ of $C$ is defined as
$\{ x \in \ZZ_{k}^n \mid x \cdot y = 0$ for all $y \in C\}$,
under the standard inner product $x \cdot y$.

If $C$ is a self-dual $\ZZ_k$-code of length $n$,
then the following lattice
\[
A_{k}(C) = \frac{1}{\sqrt{k}}
\{(x_1,\ldots,x_n) \in \ZZ^n \mid
(x_1 \bmod k,\ldots,x_n \bmod k)\in C\}
\]
is a unimodular lattice in dimension $n$.
This construction of lattices is called Construction A.

\section{From self-dual $\ZZ_4$-codes}\label{sec:4}
From now on,
we omit the term odd for odd unimodular lattices  in dimension $36$,
since all unimodular lattices in dimension $36$ are odd.
In this section, we construct ten new non-isomorphic extremal 
unimodular lattices in dimension $36$ from self-dual $\ZZ_4$-codes by  Construction A.
Five new non-isomorphic
unimodular lattices in dimension $36$ with minimum norm $3$
having shadows of minimum norm $5$ are also constructed.

\subsection{Extremal unimodular lattices}
Every $\ZZ_4$-code $C$ of length $n$ has two binary codes 
$C^{(1)}$ and $C^{(2)}$ associated with $C$:
\[
C^{(1)}= \{ c \bmod 2 \mid  c \in C \} \text{  and }
C^{(2)}= \left\{ c \bmod 2 \mid c \in \ZZ_4^n, 2c\in C \right\}.
\]
The binary codes $C^{(1)}$ and $C^{(2)}$ are called the 
{residue} and {torsion} codes of $C$, respectively.
If $C$ is a self-dual $\ZZ_4$-code, then $ C^{(1)}$ is a binary 
doubly even code with $C^{(2)} = {C^{(1)}}^{\perp}$~\cite{Z4-CS}.
Conversely,
starting from a given binary doubly even code $B$,
a method for construction of all
self-dual $\ZZ_4$-codes $C$ with $C^{(1)}=B$
was given in~\cite[Section~3]{Z4-PLF}.

The {Euclidean weight} of a codeword $x=(x_1,\ldots,x_n)$ of $C$ is
$m_1(x)+4m_2(x)+m_3(x)$, where $m_{\alpha}(x)$ denotes
the number of components $i$ with $x_i=\alpha$ $(\alpha=1,2,3)$.
The {minimum Euclidean weight} $d_E(C)$ of $C$ is the smallest Euclidean
weight among all nonzero codewords of $C$.
It is easy to see that 
$\min\{d(C^{(1)}),4d(C^{(2)})\} \le d_E(C)$,
where $d(C^{(i)})$ denotes the minimum weight of $C^{(i)}$
$(i=1,2)$.
In addition, $d_E(C) \le 4d(C^{(2)})$ and
$A_4(C)$ has minimum norm $\min\{4,d_E(C)/4\}$ (see e.g.~\cite{H11}).
%
%
Hence, if there is a binary doubly even code $B$ of length $36$
satisfying the conditions (\ref{eq:C1}) and (\ref{eq:C2}),
then
an extremal unimodular lattice in dimension $36$ is constructed
as $A_4(C)$, through a self-dual $\ZZ_4$-code $C$ with $C^{(1)}=B$.
If there is a binary $[36,k]$ code $B$ 
satisfying the conditions (\ref{eq:C1}) and (\ref{eq:C2}),
then $k=7$ or $8$ (see~\cite{Brouwer-Handbook}). 

\begin{lem}\label{lem:WE}
Let $B$ be a  binary doubly even $[36,7]$ code
satisfying the conditions {\rm (\ref{eq:C1})} and {\rm (\ref{eq:C2})}.
Then the weight enumerator of $B$
is $1+ 63 y^{16}+ 63 y^{20}+ y^{36}$.
\end{lem}
\begin{proof}
The weight enumerator of $B$ is written as:
\[
W_{B}(y)=
1
+a y^{16}
+b y^{20}
+c y^{24}
+d y^{28}
+e y^{32}
+(2^7-1-a-b-c-d-e) y^{36},
\]
where $a,b,c,d$ and $e$ are nonnegative integers.
By the MacWilliams identity, 
the weight enumerator of $B^\perp$
is given by:
\begin{align*}
W_{B^\perp}(y)=&
1
+\frac{1}{16}(- 567 + 5a + 4b + 3c + 2d + e) y
\\&
+\frac{1}{2}(1260 -10a - 10b - 9c - 7d - 4e) y^2
\\&
+\frac{1}{16}(- 112455 + 885a + 900b + 883c + 770d + 497e)y^3
+ \cdots.
\end{align*}
Since $d(B^\perp) \ge 4$, 
the weight enumerator of $B$
is written using $a$ and $b$:
\begin{align*}
W_{B}(y)=&
1 + a y^{16} + b y^{20} + (882 -10a - 4b) y^{24} 
+ (- 1638 + 20a + 6b) y^{28} 
\\ &
+ (1197 -15a - 4b) y^{32} 
+ (- 314 + 4a + b)y^{36}.
\end{align*}

Suppose that $B$ does not contain the all-one vector $\allone$.
Then $b=314 -4a$.
In this case, since the coefficients of $y^{24}$ and $y^{28}$
are $- 374 + 6a$ and $246 -4a$, 
these yield that $a \ge 62$ and $a \le 61$, respectively, 
which gives the contradiction.
Hence, $B$ contains $\allone$.
Then $b=315-4a$.
Since the coefficient $a - 63$ of $y^{32}$
is $0$, the weight enumerator of $B$
is uniquely determined as $1+ 63 y^{16}+ 63 y^{20}+ y^{36}$.
\end{proof}

\begin{rem}
A similar approach shows that
the weight enumerator of a binary doubly even $[36,8]$ 
code $B$ 
satisfying the conditions {\rm (\ref{eq:C1})} and {\rm (\ref{eq:C2})} 
is uniquely determined as
$1 + 153 y^{16} + 72 y^{20} + 30 y^{24}$.
\end{rem}

It was shown in~\cite{PST} that there are four inequivalent
binary $[36,7,16]$ codes containing $\allone$.
The four codes are doubly even.
Hence,
there are exactly four binary doubly even $[36,7]$ codes
satisfying the conditions {\rm (\ref{eq:C1})} and {\rm (\ref{eq:C2})}, 
up to equivalence.
%
%
The four codes are optimal in the sense that
these codes achieve the Gray--Rankin bound, and
the codewords of weight $16$
are corresponding to quasi-symmetric SDP 
$2$-$(36,16,12)$ designs~\cite{JT}.
Let $B_{36,i}$ be the binary doubly even $[36,7,16]$ code
corresponding to the quasi-symmetric SDP 
$2$-$(36,16,12)$ design, which is the residual design of 
the symmetric SDP $2$-$(64,28,12)$ design $D_i$
in~\cite[Section~5]{PST} $(i=1,2,3,4)$.
As described above,
all self-dual $\ZZ_4$-codes $C$ with $C^{(1)}=B_{36,i}$
have $d_E(C) =16$  $(i=1,2,3,4)$.
Hence, $A_4(C)$ are extremal.

\begin{table}[thbp]
\caption{Extremal unimodular lattices in dimension $36$}
\label{Tab:L}
\begin{center}
{\small
\begin{tabular}{l|c|c|c} 
\noalign{\hrule height0.8pt}
\multicolumn{1}{c|}{Lattices $L$} & $\tau(L)$
& $\{n_1(L),n_2(L)\}$ &  $\#\Aut(L)$ \\
\hline
Sp4(4)D8.4 in~\cite{lattice-database} &42840& $\{480, 480\}$ & 31334400 \\
$G_{36}$ in~\cite[Table~2]{G04} &42840& $\{144, 816\}$ & 576 \\ 
$N_{36}$ in~\cite{H11} &42840& $\{0, 960\}$ & 849346560 \\
$A_4(C_{36})$ in~\cite{H12} &51032& $\{0, 840\}$ & 660602880  \\
$A_6(C_{36,6}(D_{18}))$ in~\cite{Hodd} &42840& $\{384, 576\}$ & 288 \\
\hline
$A_4(C_{36, 1})$ in Section~\ref{sec:4}& 51032 &$\{ 0, 840\}$& 6291456 \\
$A_4(C_{36, 2})$ in Section~\ref{sec:4}& 42840 &$\{ 0, 960\}$& 6291456 \\
$A_4(C_{36, 3})$ in Section~\ref{sec:4}& 51032 &$\{ 0, 840\}$& 22020096\\
$A_4(C_{36, 4})$ in Section~\ref{sec:4}& 51032 &$\{ 0, 840\}$& 1966080 \\
$A_4(C_{36, 5})$ in Section~\ref{sec:4}& 51032 &$\{ 0, 840\}$& 1572864 \\
$A_4(C_{36, 6})$ in Section~\ref{sec:4}& 42840 &$\{ 0, 960\}$& 2621440 \\
$A_4(C_{36, 7})$ in Section~\ref{sec:4}& 42840 &$\{ 0, 960\}$& 1966080 \\
$A_4(C_{36, 8})$ in Section~\ref{sec:4}& 42840 &$\{ 0, 960\}$& 393216  \\
$A_4(C_{36, 9})$ in Section~\ref{sec:4}& 51032 &$\{ 0, 840\}$& 1376256 \\
$A_4(C_{36,10})$ in Section~\ref{sec:4}& 51032 &$\{ 0, 840\}$& 393216  \\
\hline
$A_{5}(D_{36,1})$ in Section~\ref{sec:E}&42840&$\{144, 816\}$ & 144 \\
$A_{5}(D_{36,2})$ in Section~\ref{sec:E}&42840&$\{456, 504\}$ &  72 \\
$A_{6}(D_{36,3})$ in Section~\ref{sec:E}&42840&$\{240, 720\}$ & 288 \\
$A_{6}(D_{36,4})$ in Section~\ref{sec:E}&42840&$\{240, 720\}$ & 576 \\
$A_{7}(D_{36,5})$ in Section~\ref{sec:E}&42840&$\{288, 672\}$ & 288 \\
$A_{7}(D_{36,6})$ in Section~\ref{sec:E}&42840&$\{144, 816\}$ &  72 \\
$A_{7}(D_{36,7})$ in Section~\ref{sec:E}&42840&$\{144, 816\}$ & 288 \\
$A_{9}(D_{36,8})$ in Section~\ref{sec:E}&42840&$\{384, 576\}$ & 144 \\
$A_{19}(D_{36,9})$ in Section~\ref{sec:E}&42840&$\{288, 672\}$ & 144 \\
\hline
$A_{5}(E_{36,1})$ in Section~\ref{sec:E}&42840& $\{456, 504\}$ & 72 \\
$A_{6}(E_{36,2})$ in Section~\ref{sec:E}&42840& $\{384, 576\}$ &144\\
\noalign{\hrule height0.8pt}
\end{tabular}
}
\end{center}
\end{table}

Using the method in~\cite[Section~3]{Z4-PLF},
self-dual $\ZZ_4$-codes $C$ are constructed 
from $B_{36,i}$.
Then ten extremal unimodular lattices $A_4(C_{36,i})$
$(i=1,2,\ldots,10)$ are constructed, where
$C_{36,i}^{(1)} = B_{36,2}$ $(i=1,2,3)$,
$C_{36,i}^{(1)} = B_{36,3}$ $(i=4,5,6,7)$ and
$C_{36,i}^{(1)} = B_{36,4}$ $(i=8,9,10)$.
To distinguish between the known lattices and our lattices,
we give in Table~\ref{Tab:L}
the kissing numbers $\tau(L)$,
$\{n_1(L),n_2(L)\}$ and the orders $\#\Aut(L)$ of
the automorphism groups,
where $n_i(L)$ denotes the number of vectors of norm $3$ in 
$Ne_i(L)$ defined in (\ref{eq:N}) $(i=1,2)$.
These have been calculated by {\sc Magma}.
Table~\ref{Tab:L} shows the following:

\begin{lem}\label{lem:N1}
The five known lattices and 
the ten extremal unimodular lattices $A_4(C_{36,i})$ $(i=1,2,\ldots,10)$
are non-isomorphic to each other.
\end{lem}

\begin{rem}
In this way, we have found
two more extremal unimodular lattices $A_4(C)$,
where $C$ are self-dual $\ZZ_4$-codes with $C^{(1)} = B_{36,1}$.
However, we have verified by {\sc Magma} that
the two lattices are isomorphic to 
$N_{36}$ in~\cite{H11} and $A_4(C_{36})$ in~\cite{H12}.
\end{rem}

\begin{rem}\label{rem}
For $L=A_4(C_{36,i})$ $(i=1,2,\ldots,10)$,
it follows from $\tau(L)$ and $\{n_1(L),n_2(L)\}$ that
one of the two unimodular neighbors $Ne_1(L)$ and  $Ne_2(L)$ 
defined in (\ref{eq:N}) is extremal.
We have verified by {\sc Magma} that 
the extremal one is isomorphic to $A_4(C_{36,i})$.
\end{rem}

For $i=1,2,\ldots,10$,
the code $C_{36,i}$ is equivalent to
some code $\overline{C_{36,i}}$ with generator matrix of the form:
\begin{equation}
\label{eq:g-matrix}
\left(\begin{array}{ccc}
I_{7} & A & B_1+2B_2 \\
O    &2I_{22} & 2D 
\end{array}\right),
\end{equation}
where $A$, $B_1$, $B_2$, $D$ are $(1,0)$-matrices,
$I_n$ denotes the identity matrix of order $n$
and $O$ denotes the $22 \times 7$ zero matrix.
We only list in Figure~\ref{Fig}
the 
$7 \times 29$ matrix
$
M_{i}=
\left(\begin{array}{cc}
 A & B_1+2B_2 
\end{array}\right)
$
to save space.
Note that 
$\left(\begin{array}{ccc}
O    &2I_{22} & 2D 
\end{array}\right)$
in (\ref{eq:g-matrix})
can be obtained from 
$\left(\begin{array}{ccc}
I_{7} & A & B_1+2B_2 
\end{array}\right)$
since $\overline{C_{36,i}}^{(2)} = {\overline{C_{36,i}}^{(1)}}^{\perp}$.
A generator matrix of $A_4(C_{36,i})$ is obtained from that of
$C_{36,i}$.

\begin{figure}[htbp]
\centering
{\footnotesize
\begin{align*}
M_{1}=\left(\begin{array}{c}
0 1 1 0 1 1 1 1 1 1 1 0 0 0 0 0 0 1 1 0 0 0 0 3 3 3 3 2 2\\
1 0 1 0 1 1 0 1 0 0 1 0 0 0 1 1 1 1 1 0 0 1 1 2 1 0 2 2 3\\
1 1 0 0 0 0 1 1 0 0 1 1 1 0 1 0 0 1 1 1 1 0 0 0 1 0 3 2 1\\
1 1 0 1 0 1 0 0 1 0 1 0 1 1 0 1 0 0 1 1 0 1 2 0 3 2 0 1 1\\
0 0 1 1 0 1 0 1 0 1 0 1 0 1 0 0 1 1 0 1 0 1 3 3 1 2 3 0 2\\
0 0 0 0 0 0 0 0 0 0 0 0 0 1 1 1 1 1 1 1 1 1 2 1 1 1 3 3 3\\
0 0 0 1 1 1 1 1 1 1 1 1 1 0 0 0 0 0 0 1 1 1 1 2 1 3 1 3 3
\end{array}\right) \  
&M_{2}=\left(\begin{array}{c}
0 1 1 0 1 1 1 1 1 1 1 0 0 0 0 0 0 1 1 0 0 0 2 1 1 1 1 0 0\\
1 0 1 0 1 1 0 1 0 0 1 0 0 0 1 1 1 1 1 0 0 1 1 2 1 0 0 2 3\\
1 1 0 0 0 0 1 1 0 0 1 1 1 0 1 0 0 1 1 1 1 0 0 0 1 0 1 2 3\\
1 1 0 1 0 1 0 0 1 0 1 0 1 1 0 1 0 0 1 1 0 1 2 0 3 2 2 1 1\\
0 0 1 1 0 1 0 1 0 1 0 1 0 1 0 0 1 1 0 1 0 1 3 3 1 2 3 0 0\\
0 0 0 0 0 0 0 0 0 0 0 0 0 1 1 1 1 1 1 1 1 1 2 1 1 3 1 3 1\\
0 0 0 1 1 1 1 1 1 1 1 1 1 0 0 0 0 0 0 1 1 1 1 2 1 1 3 3 1
\end{array}\right) \\
M_{3}=\left(\begin{array}{c}
0 1 1 0 1 1 1 1 1 1 1 0 0 0 0 0 0 1 1 0 0 0 2 1 3 1 3 0 0\\
1 0 1 0 1 1 0 1 0 0 1 0 0 0 1 1 1 1 1 0 0 1 1 2 1 0 2 2 3\\
1 1 0 0 0 0 1 1 0 0 1 1 1 0 1 0 0 1 1 1 1 0 0 0 1 0 1 2 1\\
1 1 0 1 0 1 0 0 1 0 1 0 1 1 0 1 0 0 1 1 0 1 2 0 3 2 0 1 1\\
0 0 1 1 0 1 0 1 0 1 0 1 0 1 0 0 1 1 0 1 0 1 3 3 1 0 1 0 2\\
0 0 0 0 0 0 0 0 0 0 0 0 0 1 1 1 1 1 1 1 1 1 2 1 3 1 1 3 1\\
0 0 0 1 1 1 1 1 1 1 1 1 1 0 0 0 0 0 0 1 1 1 1 2 3 3 3 3 1
\end{array}\right)  \  
&M_{4}=\left(\begin{array}{c}
0 1 1 0 1 1 1 1 1 1 1 0 0 1 1 1 1 1 0 1 1 0 0 2 2 0 3 1 3\\
1 0 1 0 1 1 0 1 0 0 1 0 0 1 0 0 1 1 1 0 1 0 1 2 3 1 0 3 2\\
1 1 0 0 0 0 1 1 0 0 1 1 1 0 1 1 0 1 0 0 1 0 2 1 0 1 1 0 1\\
0 0 1 1 1 0 0 1 1 0 0 0 1 1 1 0 0 1 0 1 1 1 3 0 1 0 2 2 3\\
1 1 0 1 1 0 0 0 0 1 1 1 0 1 1 0 0 0 1 1 1 0 0 0 2 1 1 1 0\\
0 0 0 0 0 0 0 0 0 0 0 0 0 1 1 1 1 1 1 1 1 1 2 1 1 1 1 3 1\\
0 0 0 1 1 1 1 1 1 1 1 1 1 0 0 0 0 0 0 1 1 1 3 3 0 1 3 1 1
\end{array}\right) \\
M_{5}=\left(\begin{array}{c}
0 1 1 0 1 1 1 1 1 1 1 0 0 1 1 1 1 1 0 1 1 0 2 2 0 2 1 3 3\\
1 0 1 0 1 1 0 1 0 0 1 0 0 1 0 0 1 1 1 0 1 0 1 2 3 1 0 3 2\\
1 1 0 0 0 0 1 1 0 0 1 1 1 0 1 1 0 1 0 0 1 0 2 1 0 1 1 0 1\\
0 0 1 1 1 0 0 1 1 0 0 0 1 1 1 0 0 1 0 1 1 1 3 0 1 0 2 2 3\\
1 1 0 1 1 0 0 0 0 1 1 1 0 1 1 0 0 0 1 1 1 0 0 2 2 1 3 1 0\\
0 0 0 0 0 0 0 0 0 0 0 0 0 1 1 1 1 1 1 1 1 1 2 1 1 1 1 3 1\\
0 0 0 1 1 1 1 1 1 1 1 1 1 0 0 0 0 0 0 1 1 1 3 3 0 1 3 1 1
\end{array}\right)  \  
&M_{6}=\left(\begin{array}{c}
0 1 1 0 1 1 1 1 1 1 1 0 0 1 1 1 1 1 0 1 1 0 2 0 0 0 3 1 3\\
1 0 1 0 1 1 0 1 0 0 1 0 0 1 0 0 1 1 1 0 1 0 1 2 3 1 0 3 2\\
1 1 0 0 0 0 1 1 0 0 1 1 1 0 1 1 0 1 0 0 1 0 2 1 0 1 1 0 3\\
0 0 1 1 1 0 0 1 1 0 0 0 1 1 1 0 0 1 0 1 1 1 3 2 1 0 2 2 3\\
1 1 0 1 1 0 0 0 0 1 1 1 0 1 1 0 0 0 1 1 1 0 0 0 2 1 1 1 0\\
0 0 0 0 0 0 0 0 0 0 0 0 0 1 1 1 1 1 1 1 1 1 2 3 1 3 1 1 3\\
0 0 0 1 1 1 1 1 1 1 1 1 1 0 0 0 0 0 0 1 1 1 3 1 0 3 3 3 3
\end{array}\right) \\
M_{7}=\left(\begin{array}{c}
0 1 1 0 1 1 1 1 1 1 1 0 0 1 1 1 1 1 0 1 1 0 2 2 0 2 3 1 3\\
1 0 1 0 1 1 0 1 0 0 1 0 0 1 0 0 1 1 1 0 1 0 1 2 3 1 0 3 0\\
1 1 0 0 0 0 1 1 0 0 1 1 1 0 1 1 0 1 0 0 1 0 2 1 0 1 1 2 1\\
0 0 1 1 1 0 0 1 1 0 0 0 1 1 1 0 0 1 0 1 1 1 3 2 1 0 2 0 1\\
1 1 0 1 1 0 0 0 0 1 1 1 0 1 1 0 0 0 1 1 1 0 0 0 2 1 1 1 2\\
0 0 0 0 0 0 0 0 0 0 0 0 0 1 1 1 1 1 1 1 1 1 2 1 1 1 1 1 3\\
0 0 0 1 1 1 1 1 1 1 1 1 1 0 0 0 0 0 0 1 1 1 3 3 0 1 3 3 3
\end{array}\right) \  
&M_{8}=\left(\begin{array}{c}
1 1 1 1 1 1 1 0 0 0 0 1 1 1 1 1 1 1 1 1 0 0 0 0 2 2 3 1 3\\
1 1 1 1 0 0 0 1 1 0 0 1 1 1 0 0 0 0 0 0 0 0 1 1 3 0 3 1 3\\
0 1 0 0 0 1 1 0 1 1 1 0 0 1 1 1 0 1 0 0 0 1 1 2 3 1 0 2 1\\
1 0 0 1 1 1 0 0 0 0 1 0 1 1 1 0 0 0 0 1 1 0 1 2 3 3 0 1 1\\
0 1 0 1 0 1 0 1 0 1 0 0 1 0 1 0 1 1 0 1 0 0 3 3 0 1 3 2 1\\
0 0 0 0 0 0 0 0 0 1 1 1 1 1 1 1 1 1 1 1 1 1 2 3 3 1 1 1 1\\
0 1 1 1 1 1 1 1 1 0 0 0 0 0 0 0 0 1 1 1 1 1 1 2 3 3 3 1 3
\end{array}\right) \\
M_{9}=\left(\begin{array}{c}
1 1 1 1 1 1 1 0 0 0 0 1 1 1 1 1 1 1 1 1 0 0 2 2 0 0 1 1 1\\
1 1 1 1 0 0 0 1 1 0 0 1 1 1 0 0 0 0 0 0 0 0 1 3 3 2 3 1 1\\
0 1 0 0 0 1 1 0 1 1 1 0 0 1 1 1 0 1 0 0 0 1 3 2 1 1 0 2 1\\
1 0 0 1 1 1 0 0 0 0 1 0 1 1 1 0 0 0 0 1 1 0 3 2 1 3 0 1 1\\
0 1 0 1 0 1 0 1 0 1 0 0 1 0 1 0 1 1 0 1 0 0 0 3 3 1 3 2 1\\
0 1 1 1 1 1 1 1 1 0 0 0 0 0 0 0 0 1 1 1 1 1 1 2 1 1 1 1 3\\
0 0 0 0 0 0 0 0 0 1 1 1 1 1 1 1 1 1 1 1 1 1 1 3 0 3 3 1 3
\end{array}\right)  \  
&M_{10}=\left(\begin{array}{c}
1 1 1 1 1 1 1 0 0 0 0 1 1 1 1 1 1 1 1 1 0 0 2 2 0 2 3 1 3\\
1 1 1 1 0 0 0 1 1 0 0 1 1 1 0 0 0 0 0 0 0 0 1 1 3 0 3 1 1\\
0 1 0 0 0 1 1 0 1 1 1 0 0 1 1 1 0 1 0 0 0 1 1 2 1 1 0 2 1\\
1 0 0 1 1 1 0 0 0 0 1 0 1 1 1 0 0 0 0 1 1 0 1 2 1 3 0 1 1\\
0 1 0 1 0 1 0 1 0 1 0 0 1 0 1 0 1 1 0 1 0 0 3 3 2 1 1 2 3\\
0 0 0 0 0 0 0 0 0 1 1 1 1 1 1 1 1 1 1 1 1 1 2 3 1 3 3 1 3\\
0 1 1 1 1 1 1 1 1 0 0 0 0 0 0 0 0 1 1 1 1 1 1 2 1 1 1 1 1
\end{array}\right)
\end{align*}
\caption{Generator matrices of $\overline{C_{36,i}}$ $(i=1,2,\ldots,10)$}
\label{Fig}
}
\end{figure}

\subsection{Unimodular lattices with long shadows}\label{sec:L}

The possible theta series of a unimodular lattice $L$
in dimension $36$ having minimum norm $3$
and its shadow are as follows:
\begin{align*}
&
1 
+ (960 - \alpha)q^3 
+ (42840 + 4096 \beta)q^4 
+ \cdots,
\\
&
\beta q  + (\alpha  - 60 \beta) q^3 
+ (3833856 - 36 \alpha + 1734 \beta)q^5
+ \cdots,
\end{align*}
respectively,
where $\alpha$ and $\beta$ are integers
with $0 \le \beta \le \frac{\alpha}{60} < 16$~\cite{H11}. 
Then the kissing number of $L$ is at most $960$ and 
$\min(S(L)) \le 5$.
Unimodular lattices $L$ with $\min(L)=3$ and $\min(S(L))=5$
are often called
unimodular lattices with long shadows (see~\cite{NV03}).
Only one unimodular lattice $L$ in dimension $36$
with $\min(L)=3$ and $\min(S(L))=5$ was known, 
namely, $A_4(C_{36})$ in~\cite{H11}.

Let $L$ be one of $A_4(C_{36,2})$, $A_4(C_{36,6})$, $A_4(C_{36,7})$ 
and $A_4(C_{36,8})$.
Since $\{n_1(L),n_2(L)\}=\{0,960\}$, 
one of the two unimodular neighbors $Ne_1(L)$ and  $Ne_2(L)$
in (\ref{eq:N})
is extremal and
the other is a unimodular lattice $L'$ with minimum norm $3$
having shadow of minimum norm $5$.
We denote such lattices $L'$
constructed from $A_4(C_{36,2})$, $A_4(C_{36,6})$, $A_4(C_{36,7})$ 
and $A_4(C_{36,8})$ by
$N_{36,1}$, $N_{36,2}$,  $N_{36,3}$ and $N_{36,4}$, respectively.
We list in Table~\ref{Tab:LS} the orders 
$\#\Aut(N_{36,i})$ $(i=1,2,3,4)$ of the automorphism groups, 
which have been calculated by {\sc Magma}.
Table~\ref{Tab:LS} shows the following:

\begin{prop}\label{prop:longS}
There are at least $5$ non-isomorphic
unimodular lattices $L$ in dimension $36$
with $\min(L)=3$ and $\min(S(L))=5$. 
\end{prop}

\begin{table}[th]
\caption{$\#\Aut(N_{36,i})$ $(i=1,2,3,4)$}
\label{Tab:LS}
\begin{center}
{\small
\begin{tabular}{c|r} 
\noalign{\hrule height0.8pt}
Lattices $L$ & \multicolumn{1}{c}{$\#\Aut(L)$} \\
\hline
$A_4(C_{36})$ in~\cite{H11} & 1698693120\\
$N_{36,1}$ & 12582912\\
$N_{36,2}$ & 5242880 \\
$N_{36,3}$ & 3932160 \\
$N_{36,4}$ & 786432 \\
\noalign{\hrule height0.8pt}
\end{tabular}
}
\end{center}
\end{table}

\section{From self-dual $\ZZ_k$-codes $(k \ge 5)$}\label{sec:E}

In this section, we construct more extremal 
unimodular lattices in dimension $36$ from self-dual $\ZZ_k$-codes 
$(k \ge 5)$.

Let $A^T$ denote the transpose of a matrix $A$.
An  $n \times n$ matrix is {negacirculant} if
it has the following form:
\[
\left( \begin{array}{ccccc}
r_0     &r_1     & \cdots &r_{n-1} \\
-r_{n-1}&r_0     & \cdots &r_{n-2} \\
-r_{n-2}&-r_{n-1}& \cdots &r_{n-3} \\
\vdots  & \vdots && \vdots\\
-r_1    &-r_2    & \cdots&r_0
\end{array}
\right).
\]
Let $D_{36,i}$ $(i=1,2,\ldots,9)$ and 
$E_{36,i}$ $(i=1,2)$ 
be $\ZZ_k$-codes of length $36$ with 
generator matrices  of the following form:
\begin{equation} \label{eq:GM}
\left(
\begin{array}{ccc@{}c}
\quad & {\Large I_{18}} & \quad &
\begin{array}{cc}
A & B \\
-B^T & A^T
\end{array}
\end{array}
\right),
\end{equation}
where 
$k$ are listed in Table~\ref{Tab:Codes},
$A$ and $B$ are $9 \times 9$ negacirculant matrices
with first rows $r_A$ and $r_B$ listed in Table~\ref{Tab:Codes}.
It is easy to see that these codes are self-dual since 
$AA^T+BB^T=-I_9$.
Thus, $A_k(D_{36,i})$ $(i=1,2,\ldots,9)$ and
$A_k(E_{36,i})$ $(i=1,2)$ 
are unimodular lattices, for $k$ given in Table~\ref{Tab:Codes}.
In addition, we have verified by {\sc Magma} that 
these lattices are extremal.

\begin{table}[thb]
\caption{Self-dual $\ZZ_k$-codes of length 36}
\label{Tab:Codes}
\begin{center}
{\small
\begin{tabular}{c|c|l|l}
\noalign{\hrule height0.8pt}
Codes & $k$ & \multicolumn{1}{c|}{$r_A$} &\multicolumn{1}{c}{$r_B$} \\
\hline
$D_{36,1}$& 5&$(0, 0, 0, 1, 2, 2, 0, 4, 2)$ & $(0, 0, 0, 0, 4, 3, 3, 0, 1)$\\
$D_{36,2}$& 5&$(0, 0, 0, 1, 3, 0, 2, 0, 4)$ & $(3, 0, 4, 1, 4, 0, 0, 4, 4)$\\
$D_{36,3}$&6 &$(0,1,5,3,2,0,3,5,1)$&$(3,1,0,0,5,1,1,1,1)$ \\
$D_{36,4}$&6 &$(0,1,3,5,1,5,5,4,4)$&$(4,0,3,2,4,5,5,2,4)$ \\
$D_{36,5}$& 7&$(0, 1, 6, 3, 5, 0, 4, 5, 4)$ & $(0, 1, 6, 3, 5, 2, 1, 5, 1)$\\
$D_{36,6}$& 7&$(0, 1, 1, 3, 2, 6, 1, 4, 6)$ & $(0, 1, 4, 0, 5, 3, 6, 2, 0)$\\
$D_{36,7}$& 7&$(0, 0, 0, 1, 5, 5, 5, 1, 1)$ & $(0, 5, 4, 2, 5, 1, 1, 3, 6)$\\
$D_{36,8}$& 9&$(0, 0, 0, 1, 0, 4, 3, 0, 0)$ & $(0, 4, 1, 5, 3, 5, 1, 7, 0)$\\
$D_{36,9}$&19&$(0, 0, 0, 1, 15, 15, 9, 6, 5)$
              &$(14, 16, 0, 14, 15, 8, 8, 3, 12)$ \\
\hline
$E_{36,1}$&5 &$(0, 1, 0, 2, 1, 3, 2, 2, 0)$&$(2, 0, 1, 0, 1, 1, 2, 3, 1)$\\
$E_{36,2}$&6 &$(0, 1, 5, 3, 4, 4, 1, 1, 0)$&$(4, 0, 1, 3, 4, 2, 3, 0, 1)$\\
\noalign{\hrule height0.8pt}
\end{tabular}
}
\end{center}
\end{table}



To distinguish between the above eleven lattices and 
the known $15$ lattices,
in Table~\ref{Tab:L} we give 
$\tau(L)$, $\{n_1(L),n_2(L)\}$ and $\#\Aut(L)$,
which have been calculated by {\sc Magma}.
The two lattices have the identical 
$
(\tau(L),\{n_1(L),n_2(L)\},\#\Aut(L))
$
for each of the pairs
$(A_{5}(E_{36,1}), A_{5}(D_{36,2}))$ and
$(A_{6}(E_{36,2}),A_{9}(D_{36,8}))$.
However, we have verified by {\sc Magma} that the two lattices
are non-isomorphic for each pair.
Therefore, we have the following:

\begin{lem}\label{lem:N2}
The $26$ lattices in Table~\ref{Tab:L} are
non-isomorphic to each other.
\end{lem}

Lemma~\ref{lem:N2} establishes Proposition~\ref{main}.

\begin{rem}
Similar to Remark~\ref{rem},
it is known~\cite{H11} that
the extremal neighbor is isomorphic to $L$
for the case where $L$ is $N_{36}$ in~\cite{H11},
and we have verified by {\sc Magma} that 
the extremal neighbor
is isomorphic to $L$ 
for the case where $L$ is $A_4(C_{36})$ in~\cite{H12}.
\end{rem}

\section{Related ternary self-dual codes}\label{sec:T}
In this section, 
we give a certain short observation on 
ternary self-dual codes 
related to some extremal odd unimodular lattices in dimension $36$.

\subsection{Unimodular lattices from ternary self-dual codes}

Let $T_{36}$ be a ternary self-dual code of length $36$.
The two unimodular neighbors 
$Ne_1(A_3(T_{36}))$ and $Ne_2(A_3(T_{36}))$ given in (\ref{eq:N})
are described in~\cite{HKO} as $L_S(T_{36})$ and $L_T(T_{36})$.
In this section, we use the notation 
$L_S(T_{36})$ and $L_T(T_{36})$, 
instead of $Ne_1(A_3(T_{36}))$ and $Ne_2(A_3(T_{36}))$,
since the explicit constructions and some properties 
of $L_S(T_{36})$ and $L_T(T_{36})$ are given in~\cite{HKO}.
By Theorem~6 in~\cite{HKO} (see also Theorem~3.1 in~\cite{G04}), 
$L_T(T_{36})$ is extremal when $T_{36}$ satisfies the 
following condition (a),
and both $L_S(T_{36})$ and $L_T(T_{36})$ are extremal
when  $T_{36}$ satisfies the following condition (b):
\begin{itemize}
\item[(a)]
extremal (minimum weight $12$)
and admissible (the number of $1$'s in the components of 
every codeword of weight $36$ is even),
\item[(b)]
minimum weight $9$ and maximum weight $33$.
\end{itemize}
For each of (a) and (b), 
no ternary self-dual code satisfying the condition is currently known.

\subsection{Condition (a)}
Suppose that $T_{36}$ satisfies the condition (a).
Since $T_{36}$ has minimum weight $12$, 
$A_3(T_{36})$ has minimum norm $3$ and kissing number $72$.
By Theorem~6 in~\cite{HKO},
$\min(L_T(T_{36}))=4$ and $\min(L_S(T_{36}))=3$.
Hence, since the shadow of $L_T(T_{36})$ contains no vector
of norm $1$, 
by (\ref{eq:T1}) and (\ref{eq:T2}) 
$L_T(T_{36})$ has theta series $1 + 42840 q^4 
+1916928 q^5 + \cdots$.
It follows that $\{n_1(L_T(T_{36})),n_2(L_T(T_{36}))\}=\{72,888\}$.

By Theorem~1 in~\cite{MPS}, the possible complete weight enumerator
$W_C(x,y,z)$ of
a ternary extremal self-dual code $C$ of length $36$
containing $\allone$ is written  as
\begin{align*}
 a_{1} \delta_{36}
+a_{2} \alpha_{12}^{3} 
+a_{3} \alpha_{12}^{2} {\beta_6^2}
+a_{4} \alpha_{12} (\beta_6^2)^{2}
+a_{5} (\beta_6^2)^{3}
+a_{6} \beta_6\gamma_{18} \alpha_{12}
+a_{7} \beta_6\gamma_{18} {\beta_6^2},
\end{align*}
using some $a_i \in \RR$ $(i=1,2,\ldots,7)$,
where
$\alpha_{12}=a(a^3+8p^3)$, 
$\beta_6 =a^2-12b$, 
$\gamma_{18}=a^6-20a^3p^3-8p^6$,
$\delta_{36}=p^3(a^3-p^3)^3$
and 
$a=x^3+y^3+z^3$,
$p=3xyz$,
$b=x^3y^3+x^3z^3+y^3z^3$.
From the minimum weight, we have the following:
\begin{align*}
&
a_2 = \frac{3281}{13824} - \frac{a_1}{64}, 
a_3 = \frac{203}{4608} - \frac{9 a_1}{256},
a_4 = \frac{1763}{13824} + \frac{3 a_1}{128}, 
\\&
a_5 = -\frac{277}{13824} - \frac{a_1}{256},
a_6 = \frac{1133}{1728} + \frac{3 a_1}{64}, 
a_7 = -\frac{77}{1728} - \frac{a_1}{64}.
\end{align*}
Since $W_C(x,y,z)$ contains the term
$(15180 +  2916a_1) y^{15} z^{21}$,
if $C$ is admissible, then
\[
a_1=-\frac{15180}{2916}.
\]
Hence, the complete weight enumerator of 
a ternary admissible extremal self-dual code 
containing $\allone$ is uniquely determined,
which is listed in Figure~\ref{Fig:CWE}.

\begin{figure}[htb]
\centering
{\small
\begin{align*}
&
x^{36} + y^{36} + z^{36}
+ 78706260 x^{12} y^{12} z^{12}
\\ &
+ 682 (x^{18} y^{18} + x^{18} z^{18} + y^{18} z^{18})
+ 7019232 (x^{15} y^{15} z^6 + x^{15} y^6 z^{15} +  x^6 y^{15} z^{15})
\\ &
+ 29172 (x^{24} y^6 z^6 +  x^6 y^{24} z^6 +  x^6 y^6 z^{24})
+ 10260316 (x^{18} y^9 z^9 + x^9y^{18} z^9 + x^9 y^9 z^{18})
\\ &
+ 37995408 (x^{12} y^{15} z^9 + x^{12} y^9 z^{15} +x^{15} y^{12} z^9
  + x^{15} y^9 z^{12} +x^9 y^{12} z^{15} +x^9 y^{15} z^{12})
\\ &
+ 3924756 (x^{12} y^{18} z^6 + x^{12} y^6 z^{18} +x^{18} y^{12} z^6
  + x^{18} y^6 z^{12} +x^6 y^{12} z^{18} +x^6 y^{18} z^{12})
\\ &
+ 58344 (x^{12} y^{21} z^3 +  x^{12} y^3 z^{21} +  x^{21} y^{12} z^3
  + x^{21} y^3 z^{12} +  x^3 y^{12} z^{21} +  x^3 y^{21} z^{12})
\\ &
+ 102 (x^{12} y^{24} + x^{12} z^{24} +x^{24} y^{12} + x^{24} z^{12}
  + y^{12} z^{24} +y^{24} z^{12})
\\ &
+ 170544 (x^{15} y^{18} z^3 + x^{15} y^3 z^{18} +x^{18} y^{15} z^3
  + x^{18} y^3 z^{15} + x^3 y^{15} z^{18} +x^3 y^{18} z^{15} )
\\ &
+ 641784 (x^{21} y^6 z^9 + x^{21} y^9 z^6 + x^6 y^{21} z^9
  + x^6 y^9 z^{21} + x^9 y^{21} z^6 + x^9 y^6 z^{21})
\\ &
+ 6732 (x^{24} y^3 z^9 + x^{24} y^9 z^3 +x^3 y^{24} z^9
  + x^3 y^9 z^{24} +x^9 y^{24} z^3 +x^9 y^3 z^{24})
\end{align*}
\caption{Complete weight enumerator}
\label{Fig:CWE}
}
\end{figure}

\subsection{Condition (b)}
Suppose that $T_{36}$ satisfies the condition (b).
By the Gleason theorem (see Corollary~5 in~\cite{MPS}), 
the weight enumerator of $T_{36}$ is
uniquely determined as:
\begin{align*}
&
1 + 888 y^9 + 34848 y^{12} + 1432224 y^{15} 
+ 18377688 y^{18} + 90482256 y^{21} 
\\&
+ 162551592 y^{24} 
+ 97883072 y^{27} + 16178688 y^{30} + 479232 y^{33}.
\end{align*}
By Theorem~6 in~\cite{HKO} (see also Theorem~3.1 in~\cite{G04}), 
$L_S(T_{36})$ and $L_T(T_{36})$
are extremal.
Hence, 
$\min(A_3(T_{36}))=3$ and $\min(S(A_3(T_{36})))=5$.

Note that a unimodular lattice $L$ contains a $3$-frame
if and only if  $L\cong A_3(C) $ for some ternary self-dual code $C$.
Let $L_{36}$ be any of the five lattices given in Table~\ref{Tab:LS}.
Let $L_{36}^{(3)}$ be the set $\{\{x,-x\}\mid (x,x)=3, x \in L_{36}\}$.
We define the simple undirected graph $\Gamma(L_{36})$, whose set of vertices
is the set of $480$ pairs in $L_{36}^{(3)}$ and 
two vertices $\{x,-x\},\{y,-y\}\in L_{36}^{(3)}$ are adjacent 
if $(x,y)=0$.
It follows that the $3$-frames in $L_{36}$ are precisely the $36$-cliques
in the graph $\Gamma(L_{36})$.
We have verified by {\sc Magma} that 
$\Gamma(L_{36})$ are regular graphs with valency $368$, and
the maximum sizes of cliques in $\Gamma(L_{36})$ are $12$.
Hence, 
none of  these lattices 
is constructed from some ternary self-dual code by Construction A.

\bigskip
\noindent {\bf Acknowledgments.}
The author would like to thank Masaaki Kitazume for 
bringing the observation in Section~\ref{sec:T}
to the author's attention.
This work is supported  by
JSPS KAKENHI Grant Number 23340021.



\end{document}